\documentclass[12pt, 14paper,reqno]{amsart}
\usepackage{amsmath,amsfonts,amssymb}
\usepackage{mathrsfs}
\usepackage{longtable}
\usepackage[breaklinks]{hyperref}
\usepackage{graphicx}

\theoremstyle{plain}
\newtheorem{thm}{Theorem}[section]
\newtheorem{lem}{Lemma}[section]

\newtheorem{remark}{Remark}[section]

\newtheorem{thma}{Theorem}

\theoremstyle{proof}
\makeatletter
\@namedef{subjclassname@2020}{%
  \textup{2020} Mathematics Subject Classification}
\makeatother

\numberwithin{equation}{section}

\begin{document} 
\title[Class numbers of imaginary quadratic fields]{Lehmer sequence approach to the divisibility of class numbers of imaginary quadratic fields}
\author{Kalyan Chakraborty and Azizul Hoque}
\address{Kalyan Chakraborty @Kerala School of Mathematics, Kunnamangalam, Kozhikode-673571, Kerala, India}
\email{kalychak@ksom.res.in}
\address{Azizul Hoque @Department of Mathematics, Faculty of Science, Rangapara College, Rangapara-784505, Sonitpur, Assam, India}
\email{ahoque.ms@gmail.com}


\date{\today}


\begin{abstract} Let $k\geq 3$ and $n\geq 3$ be odd integers, and let $m\geq 0$ be any integer. For a prime number $\ell$, we prove that the class number of the imaginary quadratic field $\mathbb{Q}(\sqrt{\ell^{2m}-2k^n})$ is either divisible by $n$ or by a specific divisor of $n$. Applying this result, we construct an infinite family of certain tuples of imaginary quadratic fields of the form $$\left(\mathbb{Q}(\sqrt{d}), \mathbb{Q}(\sqrt{d+1}), \mathbb{Q}(\sqrt{4d+1}), \mathbb{Q}(\sqrt{2d+4}), \mathbb{Q}(\sqrt{2d+16}), \cdots, \mathbb{Q}(\sqrt{2d+4^t})  \right)$$ with $d\in \mathbb{Z}$ and $1\leq 4^t\leq 2|d|$ whose class numbers are all divisible by $n$. Our proofs use some deep results about primitive divisors of Lehmer sequences.  
\end{abstract}
\maketitle{}
\vspace{-1mm}
\section{Introduction}
For an integer $d$, let $h(d)$ denote the class number of the quadratic field $\mathbb{Q}(\sqrt{d})$. Assume that $\ell, k$ and $n$ are positive integers such that 
\begin{equation}\label{eqi1}
\ell^2-\lambda^2k^n=-ds^2,~~~ \gcd(\ell, k)=1, \ell^2<\lambda^2y^n, ~~ \lambda\in\{1,\sqrt{2}, 2\}.
\end{equation} 
Here $d>1$ is a square-free integer and $s$ is a positive integer. One of the basic problems in the theory of quadratic fields is the divisibility of $h(-d)$, which is useful for understanding the structure of the class group of $\mathbb{Q}(\sqrt{-d})$. We direct the reader to the papers  \cite{AC55, BCH19, RM98,  NA22, SO00} (resp. \cite{CH20, HC19}  and \cite{GR78, AH21,  RM98, YA70}), when $\lambda=1$ (resp. $\lambda=\sqrt{2}$ and $\lambda=2$) for more information. 

Throughout this paper, we assume that  $k $ and $n$ are odd positive integers, $\ell$ is a prime and $m$ is a non-negative integer. The purpose of this paper is two-fold. One of its aims is to deeply investigate the $n$-divisibility of class numbers of the family of imaginary quadratic fields $\mathbb{Q}(\sqrt{\ell^{2m}-2k^n})$. The following results on the $n$-divisibility of class numbers of this family of imaginary quadratic fields were proved in \cite{CH20, KS21}. 
\begin{thma}[{\cite[Theorem 1]{CH20}}]\label{thmch}
Let $p$ and $q$ be distinct odd primes, and let $n\geq 3$ be an odd integer with $p^2 < 2q^n$ and $2q^n -p^2\ne \square$. Assume that $3q^{n/3}\ne p + 2$ whenever $3\mid n$. Then the class group of $\mathbb{Q}(\sqrt{p^2-2q^n})$ has an element of order $n$.
\end{thma}

\begin{thma}[{\cite[Theorem 3]{KS21}}]\label{thmks}
For prime numbers $p,q \geq 3$ and $k= q^r, r \in\mathbb{N}$ such that $\mathbb{Q}(\sqrt{1-2k^p})\ne \mathbb{Q}(i)$, the class number of $\mathbb{Q}(\sqrt{1 - 2k^p})$ is divisible by $p$. Moreover, this family of fields has infinitely many members each with class number divisible by $p$.  
\end{thma}
\begin{remark}
The equation $x^2+1=2y^n$ has no integer solution for odd $y>1 $ and odd $ n>1$ (see \cite[p. 65]{BS01}). Thus,  the condition `$\mathbb{Q}(\sqrt{1-2k^p})\ne \mathbb{Q}(i)$' is not required in Theorem \ref{thmks}.   
\end{remark}
Here, we prove the following general result.  
\begin{thm}\label{thm}
Let $n\geq 3$ and $k\geq 3$ be odd integers and $m\geq 0$ any integer. Assume that $\ell$ is a prime number such that $\gcd(k, \ell)=1$ and $\ell^{2m}<2k^n$. If $-d$ is the square-free part of $\ell^{2m}-2k^n$, then the following statements hold. 
\begin{itemize}
\item[(i)] If $m=0$, then $h(-d)\equiv 0\pmod n$.
\item[(ii)] If $\ell=2$ and $m\geq 1$, then \vspace{1mm}\\
$h(-d)\equiv\begin{cases} 
0\pmod n \hspace{1mm} \text{ if } 3\nmid n, \text{ or } k\equiv 1\pmod 4,\\
0 \pmod n \hspace{1mm} \text{ if } k\equiv 3\pmod 4, 3\mid n \text{ and } \frac{\ell^{2m}-2k^{\frac{n}{3}}}{-d}\ne\square,\\
0 \pmod {\frac{n}{3^j}} \hspace{1mm} \text{ if } k\equiv 3\pmod 4, 3^{j+1}\mid n \text{ and } \frac{\ell^{2m}-2k^{\frac{n}{3^{j+1}}}}{-d}\ne\square, j=0,1,\cdots.\\
\end{cases}$\vspace{1mm}
\item[(iii)] If $\ell>2$, $m\geq 1$ and $d>1$, then for a divisor $t$ of $n$, we have\vspace{1mm} \\
$h(-d)\equiv \begin{cases}
0\pmod n \text{ if } k\equiv 3\pmod 4,\\
0\pmod n \text{ if } k\equiv 1 \pmod4, \frac{1-2k^{\frac{n}{t}}}{-d}\ne\square ~~ (\text{resp. } \frac{\ell^{2m}-2k^{\frac{n}{t}}}{-d}\ne \square) \text{ with } \\
\hspace{21mm}t\ne 3~~ (\text{resp. }t=3),\\
0\pmod {\frac{n}{t}}  \text{ if } k\equiv 1 \pmod4,\frac{1-2k^{\frac{n}{t}}}{-d}=\square ~~ (\text{resp. } \frac{\ell^{2m}-2k^{\frac{n}{t}}}{-d}= \square) \text{ with } \\
\hspace{21mm} t\ne 3~~(\text{resp. }t=3).\\
\end{cases}$
\end{itemize}
\end{thm}
\begin{remark} \label{rmk2}We note down some important comments on Theorem \ref{thm}. 
\begin{itemize}
\item[(I)] The class number of  $\mathbb{Q}(\sqrt{1-2k^n})$ is divisible by any odd integer $n\geq 3$ for any odd integer $k>1$. This gives a general version of Theorem \ref{thmks}. This family of fields has infinitely many members each with class number divisible by $n$.  
\item[(II)]  The class number of the imaginary quadratic  field $\mathbb{Q}(\sqrt{2^{2m}-2k^n})$ is divisible by any odd integer $n>1$ when $k\equiv 1\pmod 4$ or $3\nmid n$. It is not hard to prove that there are infinitely many such fields with $n$-divisibility property. The proof of this fact goes along the same lines of that of Theorem 1.2 in \cite{CH18}. 
\end{itemize}
\end{remark}

In \cite{IZ18}, Iizuka proved that there are infinitely many pairs of imaginary quadratic fields $\mathbb{Q}(\sqrt{d})$ and $\mathbb{Q}(\sqrt{d+1})$ with $d \in \mathbb{Z}$ whose class numbers are divisible by $3$. This result has been extended to certain triples of imaginary quadratic fields by Chattopadhyay and Muthukrishnan \cite{CM21}. On the other hand,  Krishnamoorthy and Pasupulati \cite{KS21} proved that there are infinitely many pairs of imaginary quadratic fields $\mathbb{Q}(\sqrt{d})$ and $\mathbb{Q}(\sqrt{d+1})$ with $d \in \mathbb{Z}$ whose class numbers are divisible by a given prime $p\geq 3$. Very recently in \cite{H21}, the second author constructed an infinite family of quadruples of imaginary quadratic fields of the form $\left(\mathbb{Q}(\sqrt{d}), \mathbb{Q}(\sqrt{d+1}), \mathbb{Q}(\sqrt{d+4}), \mathbb{Q}(\sqrt{d+4p^2})\right)$with $d\in\mathbb{Z}$ and $p$ a prime, whose class numbers are all divisible by a given odd integer $n\geq 3$. Another aim of this paper is to construct an infinite family of 
tuples of imaginary quadratic fields of the form {\small
$$
\left(\mathbb{Q}(\sqrt{d}), \mathbb{Q}(\sqrt{d+1}), \mathbb{Q}(\sqrt{4d+1}),  \mathbb{Q}(\sqrt{2d+4}), \mathbb{Q}(\sqrt{2d+16}), \cdots, \mathbb{Q}(\sqrt{2d+4^m})  \right)$$}with $d\in\mathbb{Z}$ and $1\leq 4^m\leq 2|d|$ whose class numbers are all divisible by a given odd integer $n\geq 3$.  More precisely, we prove Theorem \ref{thm3} in \S\ref{S5}.

The proof of Theorems \ref{thmks} and \ref{thmch} are based upon a result of Bugeaud and Shorey \cite[Theorem 2]{BS01} on positive integer solutions of a Ramanujan-Nagell type equation. This method is applicable only when $k$ is a power of a prime, and thus we can not apply the same in the proof of Theorem \ref{thm}.  Our method is based on some remarkable results of Bilu, Hanrot and Voutier \cite{BH01, VO95} on the existence of primitive divisors of Lehmer sequences.

\section{Primitive divisors of Lehmer sequences}
A   \textit{Lehmer pair} is a pair $(\alpha, \beta)$ of algebraic integers  such that $(\alpha + \beta)^2$ and $\alpha\beta$ are non-zero coprime rational integers and $\alpha/\beta$ is not a root of unity. For a Lehmer pair $(\alpha, \beta)$, one defines the corresponding sequence of \textit{Lehmer numbers} by 
$$\mathcal{L}_n(\alpha, \beta)=\begin{cases}
\dfrac{\alpha^n-\beta^n}{\alpha-\beta}, & \text{ if $n$ is odd,} \vspace{1mm}\\
\dfrac{\alpha^n-\beta^n}{\alpha^2-\beta^2}, & \text{ if $n$ is even}.
\end{cases}$$
Note that all Lehmer numbers are non-zero rational integers. A prime divisor $p$ of $\mathcal{L}_n(\alpha, \beta)$ is \textit{primitive} if $p$ does not divide $(\alpha^2-\beta^2)^2
\mathcal{L}_1(\alpha, \beta) \mathcal{L}_2(\alpha, \beta) \cdots \mathcal{L}_{n-1}(\alpha, \beta)$. We shall make use of the following classical result of  Bilu, Hanrot and Voutier \cite[Theorem 1.4]{BH01}.
 
\begin{thma}\label{thmBH}
Let $(\alpha, \beta) $ be a Lehmer pair. Then $\mathcal{L}_n (\alpha, \beta) $ has a primitive divisor for any integer $n>30$. 
\end{thma}

A Lehmer pair $(\alpha, \beta)$ such that $\mathcal{L}_n(\alpha, \beta)$ does not have a primitive divisor is called \textit{$n$-defective}.  
Two Lehmer pairs $(\alpha_1, \beta_1)$ and $(\alpha_2, \beta_2)$ satisfying $\alpha_1/\alpha_2=\beta_1/\beta_2\in \{\pm 1, \pm\sqrt{-1} \}$ are called \textit{equivalent}. 
For a Lehmer pair $(\alpha, \beta)$, assume that $a=(\alpha+\beta)^2$ and $b=(\alpha-\beta)^2$. Then $\alpha=(\sqrt{a}\pm\sqrt{b})/2$ and $\beta=(\sqrt{a}\mp\sqrt{b})/2$. The pair $(a, b)$ is called the \textit{parameters} corresponding to the Lehmer pair $(\alpha, \beta)$. We extract the next lemma  from \cite[Table 2]{BH01}, which originally appeared in  \cite[Theorem 1]{VO95}.
\begin{lem}\label{lemVO}
Let $n$ be an odd integer such that $7\leq n\leq 29$. Then, up to equivalence, all $n$-defective Lehmer pairs $(\alpha, \beta)$ with parameters $(a, b)$ are given by:
\begin{itemize}
\item[(i)] $(a, b)=(1,-7), (1, -19), (3, -5), (5, -7), (13, -3), (14, -22)$ when $n=7$;
\item[(ii)] $(a,b) = (5,-3), (7,-1),(7,-5)$ when $n=9$;
\item[(iii)] $ (a, b)=(1,-7)$ when $n=13$;
\item[(iv)] $(a,b) = (7,-1),(10,-2)$ when $n=15$. 
\end{itemize}
\end{lem}
Let $F_k$ (resp. $L_k$) denote the $k$-th term in the Fibonacci (resp. Lucas) sequence defined by $F_0=0,   F_1= 1$,
and $F_{k+2}=F_k+F_{k+1}$ (resp. $L_0=2,  L_1=1$, and $L_{k+2}=L_k+L_{k+1}$), where $k\geq 0$ is an integer. 
One gets the following lemma from \cite[Theorem 1.3; Table 4]{BH01}. 
\begin{lem}\label{lemBH}
All $5$-defective Lehmer pairs $(\alpha, \beta)$, up to equivalence,  with parameters $(a,b)$ are given by:
$$ (a, b)=\begin{cases}
(F_{k-2\varepsilon}, F_{k-2\varepsilon}-4F_k)\text{ with } k\geq 3,\\
 (L_{k-2\varepsilon}, L_{k-2\varepsilon}-4L_k)\text{ with } k\ne 1;
 \end{cases}
 $$
where  
$k\geq 0$ is any integers and $\varepsilon=\pm 1$.
\end{lem}
We need the following lemma to deal with $5$-defective Lehmer pairs. 
\begin{lem}[{\cite[Lemma 2.1]{HO20}}] \label{lemHO} For an integer $k\geq 0$, let $F_k$ (resp. $L_k$) denote the $k$-th Fibonacci (resp. Lucas) number.
 Then for $\varepsilon=\pm 1$,
\begin{itemize}
\item[(i)] $4F_k-F_{k-2\varepsilon}=L_{k+\varepsilon}$,
\item[(ii)] $4L_k-L_{k-2\varepsilon}=5F_{k+\varepsilon}$.
\end{itemize}
\end{lem}

\section{A technical lemma}
Although the following lemma is a version of \cite[Lemma 1]{BS01},  where the result of Le \cite{LE95} was modified, but the proof presents in this paper is much simpler and uses only some easy properties of quadratic fields, omitting the use of the theory of Gaussian quadratic forms occurring in Le’s proof.
\begin{lem}\label{lemBS}
Let  $-d<-3$ be square-free and not congruent to $1$ mod $4$, let $k >1$ be an integer satisfying $\gcd(k,2d) =1$ and let $x, y, z\in \mathbb{N}$ satisfy $\gcd(x, y)=1$ and 
\begin{equation}\label{eqrn}
x^2+dy^2=2k^z.
\end{equation}
Then there exist  $x_1, y_1, z_1\in \mathbb{N}$ satisfying $\gcd(x_1, y_1)=1$ and $\lambda_1, \lambda_2\in \{1, -1\}$ such that $z_1\mid z$, the ratio $t=z/z_1$ is odd and  
\begin{equation}\label{eqrev0}
\frac{x+y\sqrt{-d}}{\sqrt{2}}=\lambda_1\left(\frac{x_1 +\lambda_2 y_1\sqrt{-d}}{\sqrt{2}}\right)^t. \end{equation}
Moreover, the class group of the field $K=\mathbb{Q}(\sqrt{-d})$ contains a cyclic subgroup of order $2z_1$. In Particular, $z_1\mid h(-d)$.   
\end{lem}
\begin{proof}
Let $R$ be the ring of integers of the field $K$. Since $-d\not\equiv 1 \pmod 4$,
the discriminant of $K$ equals $-4d$ and  hence $2$ is ramified. Thus we have $2R = \mathfrak{p}^2$
with a prime ideal $\mathfrak{p}$. Denote by $N(\cdot)$ the norm of elements and ideals of $R$.

Putting $\alpha=x+y\sqrt{-d}$ and writing the equality \eqref{eqrn} in the form 
$$N(\alpha)=\alpha\bar{\alpha}=(x+y\sqrt{-d})(x-y\sqrt{-d})=x^2 +dy^2 =2k^z,$$
one observes that $\mathfrak{p}$ is the greatest common divisor of the ideals $\alpha R$ and $\bar{\alpha}R$, hence the ideals $\alpha R/\mathfrak{p}$ and $\bar{\alpha} R/\mathfrak{p}$ are coprime. This implies that each of them is a $z$-th power. In particular, we have
$$\frac{\alpha R}{\mathfrak{p}} = I^z$$
for some ideal $I$. Note that the equality
$$N(I^z) = N(\alpha)/N(\mathfrak{p}) = k^z$$
implies
\begin{equation}\label{eqrev1} N(I) = k.\end{equation}

Now let $z_1$ be the smallest positive integer such that the ideal $\mathfrak{p}I^{z_1}$ is principal, say $\mathfrak{p}I^{z_1} =\beta R$, and let $s$ the smallest positive integer such that the ideal $I^s$ is principal, say $I^s = \gamma R$. This shows that the ideal class containing $I$ is of order $s$, so the class group contains a cyclic subgroup of order $s$.

Note that $z_1 < s$, as otherwise $pI^{z_1-s}$ would be principal, contradicting the minimality of $z_1$. On the other hand, the equality
$$I^{2z_1}=(I^{z_1})^2=\left( \frac{\beta R}{\mathfrak{p}} \right)^2=\frac{\beta^2 R}{2}$$
shows the principality of $I^{2z_1}$, thus $s \mid 2z_1 < 2s$, leading to
\begin{equation}\label{rev2}
s = 2z_1. 
\end{equation}
This implies that the class group of $K$ contains cyclic subgroups of order $2z_1$ and $z_1$.

To show the divisibility of $z$ by $z_1$, we write $z=qz_1+r$ with $0 \leq r <z_1$ and $q, r\in \mathbb{Z}$. This implies
$$\alpha R=\mathfrak{p}I^z=\mathfrak{p}(I^{z_1})^qI^r=\beta^q\mathfrak{p}^{1-q}I^r.$$

If $q$ is even, then 
$$\alpha R=\beta^q2^{-q/2}\mathfrak{p}I^r,$$
hence $r > z_1$, a contradiction. Thus $q$ is odd and we get
$$\alpha R = 2^{(1-q)/2}\beta^qI^r, $$
hence $I^r$ is principal, so $s$ divides $r$. Were $r > 0$, then \eqref{rev2} would imply $z_1 < 2z_1 = s \leq  r < z_1$, a contradiction. Therefore $r = 0$ and $z_1$ divides $z$. 

Let $t = z/{z_1}$ and observe that in view of
$$\alpha R = \mathfrak{p}I^z = \mathfrak{p}I^{tz_1} = (\mathfrak{p}I^{z_1})^t\mathfrak{p}^{1-t} = \beta^t\mathfrak{p}^{1-t}$$
one has $2 \nmid t$.

Now write $\beta=u+v\sqrt{-d}$ with $u, v\in \mathbb{Z}$ and use \eqref{eqrev1} to get
$$u^2 +v^2d=N(\beta)=N(\mathfrak{p})N(I^{z_1})=2k^{z_1}$$ 
and in view of $\alpha R=\mathfrak{p}I^z$, $\beta R=\mathfrak{p} I^{z_1}$ and $2\nmid t =z/{z_1}$ one arrives at
$$\alpha R = \frac{\beta^t}{2^{(t-1)/2}},$$
which after choosing properly the signs $\lambda_1$ and $\lambda_2$ leads to \eqref{eqrev0}.
\end{proof}

\section{Proof of Theorem \ref{thm}}
Assume that $-d$ is the square-free part of $\ell^{2m}-2k^n$. Then we can write 
\begin{equation}\label{eqt1}
\ell^{2m}-2k^n=-ds^2
\end{equation}
for some positive integer $s$. This shows that $(x,y,z)=(\ell^m, s, n)$ is a positive integer solution of \eqref{eqrn} and thus by Lemma \ref{lemBS}, we have
\begin{equation}\label{eqt2}
\frac{\ell^m+s\sqrt{-d}}{\sqrt{2}}=\lambda_1\left(\frac{u+\lambda_2 v\sqrt{-d}}{\sqrt{2}}\right)^t,~~ \lambda_1, \lambda_2 \in \{-1, 1\},
\end{equation}
with 
\begin{equation}\label{eqt3}
n=n_1t,~~(n_1, t\in \mathbb{N}).
\end{equation}
Here $u$ and $v$ are positive integers satisfying
\begin{equation}\label{eqt4}
u^2+dv^2=2k^{n_1},~~ \gcd(u, v)=1 
\end{equation}
and 
\begin{equation}\label{eqt5}
n_1\mid h(-d).
\end{equation}
Now we define:
$$\begin{cases}
\alpha=\dfrac{u+\lambda_2v\sqrt{-d}}{\sqrt{2}},\\
\beta=\dfrac{-u+\lambda_2v\sqrt{-d}}{\sqrt{2}}.
\end{cases}$$
Then utilizing \eqref{eqt4} we can show that both $\alpha$ and $\beta$ are algebraic integers. Again applying \eqref{eqt4}, we get $(\alpha+\beta)^2=-2v^2d$ and $\alpha\beta=-k^{n_1}$. Since $k$ is odd and $\gcd(k, vd)=1$, so that $(\alpha+\beta)^2$ and $\alpha\beta$ are coprime. 
Furthermore, it follows from the following identity 
 $$\frac{2dv^2}{k^{n_1}}=\frac{(\alpha+\beta)^2}{\alpha\beta}=\frac{\alpha}{\beta}+\frac{\beta}{\alpha}+2$$
 that 
 $$k^{n_1}\left( \frac{\alpha}{\beta}\right)^2+2(k^{n_1}-dv^2)\frac{\alpha}{\beta}+k^{n_1}=0.$$
 Since $\gcd(k^{n_1}, 2(k^{n_1}-dv^2))=\gcd(k^{n_1}, dv^2)=1$, so that $\dfrac{\alpha}{\beta}$ is not an algebraic integer and thus it not a root of unity. Therefore, $(\alpha, \beta)$ is a Lehmer pair  with parameters $(-2dv^2, 2u^2)$ and the corresponding Lehmer sequence is given by  
 $$\mathcal{L}_t(\alpha, \beta)=\frac{\alpha^t-\beta^t}{\alpha-\beta}$$
 as $t$ is odd. 
 Employing \eqref{eqt2}, we get 
 \begin{equation}\label{eqt6}
 |\mathcal{L}_t(\alpha, \beta)|=\frac{\ell^m}{u}.
 \end{equation}
 We now divide the rest of the proof into two parts depending on the values of $m$.
 \subsection*{Case I: when $m=0$}
 In this case, \eqref{eqt6} implies that $u=1$ as $\mathcal{L}_t(\alpha, \beta)$ is a rational integer. Thus $|\mathcal{L}_t(\alpha, \beta)|=1$ and hence the Lehmer number $\mathcal{L}_t(\alpha, \beta)$ has no primitive divisor. Therefore by Theorem \ref{thmBH}, we get $t<30$.  Since $(-2dv^2, 2u^2)=(-2dv^2, 2)$ is the corresponding parameters, so that Lemma \ref{lemVO} gives $t\leq 5$. 
 
 In case of $t=5$, Lemma \ref{lemBH} gives $ F_{k'-2\varepsilon}-4F_{k'}=2 $ and $  L_{k'-2\varepsilon}-4L_{k'}=2$, where $k'\geq 0$ is an integer and $\varepsilon=\pm 1$. However, Lemma \ref{lemHO} ensures that none of these are possible. 

Now for $t=3$, we equate the real parts in \eqref{eqt2} to get $1-3dv^2=\pm 2$. This implies that $dv^2=1$, which is not possible. Thus  $t=1$ and hence \eqref{eqt3} and  \eqref{eqt5} together give $n\mid h(-d)$. 

\subsection*{Case II: when $m\geq 1$} By \eqref{eqt6}, we have $u=\ell^{m_1}$ for some integer $0\leq m_1\leq m$. Thus, $\left(\alpha^2-\beta^2\right)^2=-4\ell^{2m_1}v^2d$ and hence \eqref{eqt6} ensures that $\mathcal{L}_t(\alpha, \beta)$ has no primitive divisor for $m_1\geq 1$. As $(-2dv^2, 2u^2)=(-2dv^2, 2\ell^{2m_1})$ is the corresponding parameters, so that Theorem \ref{thmBH} and Lemma \ref{lemVO} together imply that $t\leq 5$. In the case of $t=5$, using Lemma \ref{lemBH} and Lemma \ref{lemHO}, $L_{k'+\varepsilon}=-2\ell^{2m_1} $ and $  F_{k'+\varepsilon}=-2\ell^{2m_1}$ with $k'\geq 0$ an integer, which are not possible. 

As $u=\ell^{m_1}$, so that equating the real parts in \eqref{eqt2} for $t=3$, we get 
\begin{equation}\label{eqt7}
\ell^{2m_1}-3dv^2=\pm 2\ell^{m-m_1}.
\end{equation}
Reading it modulo $\ell$ (resp. $4$) when $\ell>2$ (resp. $\ell=2$), we get $m=m_1$ and thus \eqref{eqt7} becomes
$$\ell^{2m}-3dv^2=\pm 2.$$
By reading it modulo $3$ we see that the `+' sign is not possible, and thus we have
\begin{equation}\label{eqt8}
\ell^{2m}+2=3dv^2.
\end{equation}
As $u=\ell^m$, so that \eqref{eqt4} becomes
\begin{equation}\label{eqt9}
\ell^{2m}-2k^{\frac{n}{3}}=-dv^2.
\end{equation}
Now adding \eqref{eqt8} and three times of \eqref{eqt9}, and then reading modulo $4$, we get 
\begin{equation}\label{eqt10}
k\equiv \begin{cases}
1\pmod 4\hspace{6mm} \text{ if } \ell>2,\\
3\pmod 4\hspace{6mm} \text{ if } \ell=2.
\end{cases}
\end{equation}
Thus if 
$$k\equiv \begin{cases}
3\pmod 4\hspace{6mm} \text{ if } \ell>2,\\
1\pmod 4\hspace{6mm} \text{ if } \ell=2,
\end{cases}
$$
then $t=1$ and hence by the equations \eqref{eqt3} and  \eqref{eqt5},  we conclude that $n\mid h(-d)$.

On the other hand, when
$$k\equiv \begin{cases}
1\pmod 4\hspace{6mm} \text{ if } \ell>2,\\
3\pmod 4\hspace{6mm} \text{ if } \ell=2,
\end{cases}
$$
we have $t=1$ provided $(\ell^{2m}-2k^{\frac{n}{3}})/-d $ is not a square in $\mathbb{N}$ by \eqref{eqt9}. Thus  the equations \eqref{eqt3} and  \eqref{eqt5} together imply that $n\mid h(-d)$.

Again, if $(\ell^{2m}-2k^{\frac{n}{3}})/-d$ is a square in $\mathbb{N}$, then repeating the above procedure we can show that $\frac{n}{3}\mid h(-d)$ provided  $(\ell^{2m}-2k^{\frac{n}{3^2}})/-d $ is not a square in $\mathbb{N}$. Thus we can conclude that $\frac{n}{3^j}\mid h(-d)$ provided  $(\ell^{2m}-2k^{\frac{n}{3^{j+1}}})/-d $ is not a square in $\mathbb{N}$ for $j=0,1,2,\cdots $.

We now consider the remaining case that is $m_1=0$. In this case, $u=1$ and thus \eqref{eqt4} gives 
\begin{equation}\label{eqt11}
1-2k^{\frac{n}{t}}=-dv^2. 
\end{equation}
If $\ell=2$, then by \eqref{eqt1}, we get $d\equiv 2\pmod 4$ which contradicts to \eqref{eqt11}. Thus $\ell>2$ and hence \eqref{eqt1} and \eqref{eqt11} together give $$k^t+dv^2k^t=\ell^{2m}+ds^2.$$
Reading it modulo $8$, we get $k(d+1)\equiv 1+d\pmod 8$. Since $d+1\equiv \pm 2\pmod 8$ (using \eqref{eqt1}), so that the last congruence implies that $k\equiv 1\pmod 4$. Thus if $k\equiv 3\pmod 4$, then as before we can conclude that $n\mid h(-d)$.  

Similarly from \eqref{eqt11}, we can conclude that $n\mid h(-d)$ when $(1-2k^{\frac{n}{t}})/-d$ is not a square; otherwise  by Case I, $\frac{n}{t}\mid h(-d)$. 
This completes the proof. 
\section{Tuples of imaginary quadratic fields }\label{S5}
Krishnamoorthy and Pasupulati recently proved the following result using Theorem \ref{thmks}. 
\begin{thma}[{\cite[Theorem 4]{KS21}}]\label{thmks2}
For every odd prime number $p \geq 3$, there is an infinite family of pairs of imaginary quadratic fields $\mathbb{Q}(\sqrt{d})$ and $\mathbb{Q}(\sqrt{d + 1})$ with $d \in\mathbb{Z}$ whose class numbers are both divisible by $p$.
\end{thma}
Applying Remark \ref{rmk2}, we get the following generalization of Theorem \ref{thmks2}. 
\begin{thm}\label{thmc1}
For every odd integer number $n \geq 3$, there is an infinite family of pairs of imaginary quadratic fields $\mathbb{Q}(\sqrt{d})$ and $\mathbb{Q}(\sqrt{d + 1})$ with $d \in\mathbb{Z}$ whose class numbers are both divisible by $n$.
\end{thm}
Note that the imaginary quadratic fields $\mathbb{Q}(\sqrt{d})$ and $\mathbb{Q}(\sqrt{d + 1})$ in Theorem \ref{thmc1} are explicitly known (see Remark \ref{rmk2}). On the other hand, Xie and Chao \cite[Theorem 1.2]{XC20} proved 
that for an odd positive integer $n$ and a positive integer $m$, there are infinitely many pairs of imaginary fields $\mathbb{Q}(\sqrt{d})$ and $\mathbb{Q}(\sqrt{d + m})$ whose class numbers are both divisible by $n$.
Very recently using rational points on an elliptic curve with positive Mordell-Weil rank to parametrize imaginary quadratic fields, Iizuka, Konomi and Nakano \cite[Theorem 2]{IKN21} proved that for any given non-zero integers $a, b, a_1$ and $b_1$, there is an infinite family of pairs of distinct quadratic fields $\mathbb{Q}(\sqrt{ad+b})$ and $\mathbb{Q}(\sqrt{a_1d+b_1})$ such that both class numbers are divisible by $3$, $5$ or $7$. 

Here, we utilize Remark \ref{rmk2} to construct an infinite family of certain tuples of imaginary quadratic fields of the form $$\left(\mathbb{Q}(\sqrt{d}), \mathbb{Q}(\sqrt{d+1}), \mathbb{Q}(\sqrt{4d+1}), \mathbb{Q}(\sqrt{2d+4}), \mathbb{Q}(\sqrt{2d+16}), \cdots,  \mathbb{Q}(\sqrt{2d+4^m})\right)$$ with $1\leq 4^m\leq 2|d|$ whose class numbers are all divisible by an odd integer $n\geq 3$. More precisely we prove: 
 
\begin{thm}\label{thm3}
If $k\geq 3, n\geq 3$ are odd integers and $d=(1-2k^n)^n$, then the class number of $\mathbb{Q}(\sqrt{D})$ is divisible by $n$ for $D\in \{d, d+1, 4d+1\}\cup\{2d+4^m : 1\leq 4^m\leq 2|d|\}$. 
\end{thm}
%
%
  
We need the following two results to prove Theorem \ref{thm3}. 
\begin{thma}[{\cite[Corollary 1]{CO03}}]\label{thmco} 
For  any odd integers  $n\geq 3$ and $a\geq 3$, the class number of the imaginary quadratic field $\mathbb{Q}(\sqrt{1-a^n})$ is divisible by $n$, except for $(n,a)=(5,3)$.
\end{thma}

\begin{thma}[{\cite[Theorem 1]{LO09}}]\label{thmlo}
For   any odd integer $n\geq 3$ and any integer $b\geq  2$, the ideal class groups of the imaginary quadratic fields $\mathbb{Q}(\sqrt{1-4b^n})$ contain an element of order $n$.
\end{thma} 
\subsection*{Proof of Theorem \ref{thm3}}
Let $n\geq 3$ be an odd integer. For any odd integer $k\geq 3$, define  the negative integer $d$ by $$d:=(1-2k^n)^n.$$
Then by Remark \ref{rmk2}, we have 
$$n\mid h(1-2k^n)=h(d).$$
Now we set
$$a:=2k^n-1.$$
Then by $k\geq 3$ and $n\geq 3$, we have $a\geq 53 $ and $$d+1=(1-2k^n)^n+1=(-a)^n+1=1-a^n.$$
 Therefore by Theorem \ref{thmco}, we have $n\mid h(d+1)$.  
 
 For any integer $m\geq 0$, we have \begin{align*}
 & 4d+1=4(1-2k^n)^n+1=4(-a)^n+1=1-4a^n,\\
& 2d+4^m=2(1-2k^n)^n+4^m=2(-a)^n+4^m=2^{2m}-2a^n.
\end{align*}
It follows from Theorem \ref{thmlo} and Remark \ref{rmk2} that $n\mid h(4d+1)$ and $n\mid h(2d+4^m)$, respectively provided $1\leq 4^m<2|d|$. 


\section*{Acknowledgement}
The authors are grateful to Professor Yasuhiro Kishi for his valuable suggestion to improve the presentation of the paper. The authors are thankful to Professor Y. Iizuka for providing a copy of \cite{IKN21}.  The authors gratefully acknowledge the anonymous referee for his/her valuable suggestion that immensely improved the presentation of the paper. 
This work was supported by SERB MATRICS Project (No. MTR/2021/000762), Govt. of India.


\end{document}